\documentclass[11pt]{amsart}
\usepackage[colorlinks, citecolor=blue, dvipdfm, pagebackref]{hyperref}
\usepackage{extarrows}

\newtheorem{theorem}{Theorem}[section]
\newtheorem{lemma}[theorem]{Lemma}

\theoremstyle{definition}
\newtheorem{definition}[theorem]{Definition}

\newtheorem{example}[theorem]{Example}

\newtheorem{proposition}[theorem]{Proposition}
\newtheorem{corollary}[theorem]{Corollary}
\newtheorem{remark}[theorem]{Remark}

\theoremstyle{remark}

\newcommand{\be}{\begin{equation}}
\newcommand{\ee}{\end{equation}}

\numberwithin{equation}{section}

\begin{document}

\title{Nonnegative Hermitian vector bundles and Chern numbers}

\author{Ping Li}
\address{School of Mathematical Sciences, Tongji University, Shanghai 200092, China}
%\address{Department of Mathematics, Faculty of Science and Engineering, Waseda University, Shinjuku, Tokyo 169-8555, Japan}
%\curraddr{Department of Mathematics and Information Sciences, Tokyo
%Metropolitan University, Tokyo 192-0397, Japan}
%\curraddr{Department of Mathematics and Information Sciences, Tokyo
%Metropolitan University, Tokyo 192-0397, Japan}
\email{pingli@tongji.edu.cn\\
pinglimath@gmail.com}
%\author{}
%\address{}
%\email{}
\thanks{The author was partially supported by the National
Natural Science Foundation of China (Grant No. 11722109).}

%    General info
 \subjclass[2010]{57R20, 32Q55, 53C55, 32M10, 57R22.}

%\date{January 1, 2001 and, in revised form, June 22, 2001.}

%\dedicatory{Dedicated to Professor Boju Jiang on his 80th birthday.}
\keywords{Nonnegativity, Hermitian vector bundle, Chern form, Chern class, Chern number, globally generated vector bundle, homogeneous complex manifold, Moishezon manifold, Kodaira dimension, algebraic dimension.}

\begin{abstract}
We show in this article that if a holomorphic vector bundle has a nonnegative Hermitian metric in the sense of Bott and Chern, which always exists on globally generated holomorphic vector bundles, then some special linear combinations of Chern forms are strongly nonnegative. This particularly implies that all the Chern numbers of such a holomorphic vector bundle are nonnegative and can be bounded below and above respectively by two special Chern numbers. As applications, we obtain a family of new results on compact connected complex manifolds which are homogeneous or can be holomorphically immersed into complex tori, some of which improve several classical results.
\end{abstract}

\maketitle

\section{Introduction}\label{sectionintroduction}
Throughout this article, all the vector bundles considered are holomorphic and over \emph{compact} complex manifold.

The concept of positivity/nonnegativity has played a central role in complex differential geometry and algebraic geometry. Bott and Chern (\cite{BC}) introduced a notion of nonnegativity on Hermitian vector bundles and applied it to initiate the study of high-dimensional value distribution theory. As showed in \cite{BC}, the existence of such a metric implies that all the Chern forms are strongly nonnegative and on globally generated vector bundles such metrics always exist. Later, globally generated vector bundles were investigated in detail by Matsushima and Stoll in \cite{MS} and they, among other things, related the non-vanishing of the Chern classes and Chern numbers to the transcendence degree of meromorphic functions on the base manifolds.

Soon after the appearance of \cite{BC}, the notions of ampleness and Griffiths-positivity were introduced respectively by Hartshorne (\cite{Ha}) and Griffiths (\cite{Gr}) and it turns out that Bott-Chern nonnegativity implies Griffiths-nonnegativity %(cf. Example \ref{examplebottchernimpliesnefness})
 while Griffiths-positivity implies ampleness. Griffiths raised in \cite{Gr} the question of characterizing the polynomials in the Chern classes/forms which are positive as cohomology classes/differential forms for Griffiths-positive or ample vector bundles. On the class level this was answered completely by Fulton and Lazarsfeld (\cite{FL}), extending an earlier result of Bloch and Gieseker (\cite{BG}). Indeed Fulton and Lazarsfeld showed that the set of such polynomials in the Chern classes for ample vector bundles is exactly the cone generated by Schur polynomials of the Chern classes. These inequalities of Fulton-Lazarsfeld type were showed by Demailly, Peternell and Schneider (\cite{DPS}) to remain true for numerically effective (``nef" for short) vector bundles over compact K\"{a}hler manifolds.

On the other hand,
Griffiths' question on the form level is still largely unknown. Griffiths himself showed in \cite[p. 249]{Gr} that the second Chern form is positive on a Griffiths-positive vector bundle. %Although in the proof is only for rank $2$ vector bundles over surfaces, it can be easily adapted to the general case.
Note that its proof is purely algebraic, which seems to be difficult to be generalized to higher dimensions. Recently Guler (\cite{Gu}) showed that the dual Segre forms, which are formal inverse of the total Chern forms, are positive on Griffiths-positive vector bundles using some geometric arguments.

The main purposes of this article are two-folded.
 Our \emph{first main purpose} is to show a useful technical result, Proposition \ref{theoremschur}, which says that on a Bott-Chern nonnegative Hermitian vector bundles, the cone generated by the Schur polynomials in the Chern forms are strongly nonnegative, thus providing some positive evidence to the aforementioned Griffiths' question on the form level. Our proof of Proposition \ref{theoremschur} is built on a relationship established in \cite{FL} that the cone generated by the Schur polynomials coincides to the Griffiths cone defined in \cite{Gr}. When taking some special Schur polynomials in Proposition \ref{theoremschur}, we shall see in Theorem \ref{theoremchernnumber} that the self-products of Chern forms and thus all the Chern numbers of such vector bundles can be bounded below and above by two of them respectively. As is well-known globally generated vector bundles admit Bott-Chern nonnegative metrics (see Example \ref{examplegloballygenerated}), this implies that Theorem \ref{theoremchernnumber} imposes strong constraints on the Chern numbers of such vector bundles.

\emph{Our second main purpose} is to apply Theorem \ref{theoremchernnumber} to such vector bundles and especially to compact complex manifolds whose holomorphic tangent or cotangent bundles are globally generated to yield a family of new results on them, some of which improve several related classical results. To be more precise, \emph{our first application} (Theorem \ref{theoremapp1} and Corollary \ref{coromain}) is to give lower and upper bounds for Chern numbers of various nonnegative type holomorphic vector bundles on compact complex manifolds, whose nonnegativity has been well-known for several decades. \emph{Our second application} (Theorems \ref{theoremproj2} and \ref{coroproj2}) is concerned with the projectivicity of compact connected homogeneous complex manifolds in terms of the non-vanishing of their Chern numbers. \emph{Our third application} (Theorem \ref{theoremappl3}) is to give an in-depth investigation on the structure of compact complex manifolds holomorphically immersed into complex tori in detail.

\subsection*{Organization of this article}
The rest of this article is organized as follows. We introduce some necessary notation and symbols in Section \ref{Preliminaries} and state in Section \ref{sectionmainresults} our main technical results in this article, Proposition \ref{theoremschur} and Theorem \ref{theoremchernnumber}. Then we give in Section \ref{sectionexample} some examples where the Bott-Chern nonnegativity is satisfied. In Sections \ref{sectionconsequence1}, \ref{sectionconsequence2} and \ref{sectionconsequence3} we apply our established technical results to obtain various related consequences, some of which improve several classical results. The proof of Proposition \ref{theoremschur} itself shall be given in the last section, Section \ref{sectionproofoftheorem1}.

\section*{Acknowledgements}
This paper was initiated while the author was visiting  the Max-Planck Institut f\"{u}r Mathematik in Bonn from September 2016 to February 2017. The author thanks the Institute for the hospitality and financial support. The author also thanks Xiaokui Yang for informing him the related results in \cite{Zhang} and Vamsi Pingali for some crucial comments as well as pointing out the reference \cite{BP}.

\section{Preliminaries}\label{Preliminaries}
Suppose that $(E^r,h)\longrightarrow M^n$ is a rank $r$ Hermitian vector bundle over a compact complex manifold $M$ with $\text{dim}_{\mathbb{C}}M=n$.  This means that $E^r$ is a rank $r$ holomorphic vector bundle over $M$ and equipped with a Hermitian metric $h$ on each fiber varying smoothly. We denote by $\nabla$ the Chern connection of $(E^r,h)\longrightarrow M^n$, i.e., $\nabla$ is the unique connection compatible with both the complex structure and the Hermitian metric $h$. Then the curvature tensor $R$ of $\nabla$ is given by
$$R:=\nabla^2\in A^{1,1}(M;\text{Hom}(E,E))=
A^{1,1}(M;E^{\ast}\otimes E),$$
where $A^{1,1}(M;\cdot)$ is the set of complex-valued smooth $(1,1)$-forms with values in some bundle.

If $\{e_1,\ldots,e_r\}$ is a locally defined frame field of $E$, then
$$R(e_1,\ldots,e_r)=(e_1,\ldots,e_r)(\Omega^i_j),$$
where $\Omega:=(\Omega^i_j)$ ($i$ row, $j$ column) is the curvature matrix with respect to $\{e_i\}$ whose entries are $(1,1)$-forms. If $\{\tilde{e_i}\}$ is another frame with $$(\tilde{e_1},\ldots,\tilde{e_r})=(e_1,\ldots,e_r)P,$$ then the curvature matrix $\widetilde{\Omega}$ with respect to $\{\tilde{e_i}\}$ is related to $\Omega$ by
\be\label{curvaturematrixtransformation}
\widetilde{\Omega}=P^{-1}\Omega P.\ee

It is well-known that the following $c_i(E,h)$ ($0\leq i\leq n$):
\be\label{det}\text{det}\big(tI_r+\frac{\sqrt{-1}}{2\pi}
(\Omega^i_j)\big)=:\sum_{i=0}^nc_i(E,h)\cdot t^{n-i},\qquad I_r:~\text{$r\times r$ identity matrix},\ee
are globally well-defined, real and closed $(i,i)$-forms and called the $i$-th Chern forms of $(E,h)$, which represent the Chern classes $c_i(E)$ of $E$.

The following definition was introduced by Bott and Chern (\cite[p. 90]{BC}).

\begin{definition}[Bott-Chern]\label{definitionBottChern}
A Hermitian vector bundle $(E^r,h)\longrightarrow M^n$ is called \emph{Bott-Chern nonnegative}, denoted by $h\geq_{\text{BC}}0$, if for any point of $M$, there exist a \emph{unitary} frame field around it and a matrix $A$ with $r$ rows and whose entries are $(1,0)$-forms, such that the curvature matrix $\Omega$ under this unitary frame field satisfies
\be\label{BottCherndef}\Omega=A\wedge\overline{ A^t}.\ee
Here $``t"$ denotes the transpose of a matrix.
\end{definition}

\begin{remark}~
\begin{enumerate}
%\item
%This definition is obviously inspired by the elementary fact that, for any matrix $A$ whose entries are ordinary complex numbers, the matrix $A\overline{A^t}$ is Hermitian nonnegative definite.
\item
This definition is independent of the unitary frame we choose as the transformation matrix $P$ is a unitary matrix between them and thus from $(\ref{curvaturematrixtransformation})$ we have
$$\widetilde{\Omega}=P^{-1}\Omega P=P^{-1}(A\wedge \overline{A^t})P=\overline{P^{t}}(A\wedge \overline{A^t})P=(\overline{P^t}A)\wedge\overline{(\overline{P^t}A)^t}.$$

\item
Here we don't have any requirement on the number of columns of the matrix $A$, which may vary along the choice of the points on $M$.

\item
We shall see in Example \ref{examplegloballygenerated} that globally generated vector bundles over compact complex manifolds form an important subclass of those admitting Bott-Chern nonnegative Hermitian metrics.
\end{enumerate}
\end{remark}

Before stating our first technical result, Proposition \ref{theoremschur}, we need some more notation. Following \cite{FL}, we denote by $\Gamma(i,r)$ the set of partitions $\lambda=(\lambda_1,\lambda_2,\ldots,\lambda_i)$ of weight $i$ by nonnegative integers $\lambda_j\leq r$:
$$r\geq\lambda_1\geq\lambda_2\geq\cdots\geq\lambda_i\geq0,\qquad
\sum_{j=1}^{i}\lambda_j=i.$$

For each partition $\lambda=(\lambda_1,\lambda_2,\ldots,\lambda_i)\in\Gamma(i,r)$, a Schur polynomial $$S_{\lambda}(c_1,\ldots,c_r)\in\mathbb{Z}[c_1,\ldots,c_r]$$ is attached to as follows.
\be
\begin{split}
S_{\lambda}(c_1,\ldots,c_r):=&\text{det}(c_{\lambda_j-j+k})_{1\leq j,k\leq i}\qquad(\text{$j:$ row, $k:$ column})\\
=&
%\text{det}\left(\begin{array}{cccc}
\begin{vmatrix}
c_{\lambda_1} & c_{\lambda_1+1} &\cdots &c_{\lambda_1+i-1}\\
c_{\lambda_2-1} & c_{\lambda_2} &\cdots&c_{\lambda_2+i-2}\\
\vdots&\vdots&\ddots&\vdots\\
c_{\lambda_i-i+1}&c_{\lambda_i-i+2}&\cdots&c_{\lambda_i}
%\end{array}\right)
\end{vmatrix}
.\end{split}\nonumber\ee

Here we make the convention that $c_0=1$ and $c_j=0$ if $j<0$ or $j>r$. In particular, we have
\be\label{schurspecialcase1}
S_{(i,0,\ldots,0)}(c_1,\ldots,c_r)=c_i\ee
%\be S_{(1,\ldots,1)}(c_1,\ldots,c_r)=\text{dual Serge class of $c_1,\ldots,c_r$},\nonumber\ee
and
\be\label{schurspecialcase2}
\begin{split}
&S_{(i-j,j,0,\ldots,0)}(c_1,\ldots,c_r)\\
=&\begin{vmatrix}
c_{i-j} & c_{i-j+1} &\ast &\cdots &\ast\\
c_{j-1} & c_{j} &\ast&\cdots&\ast\\
0&0&1&\cdots&\ast\\
\vdots&\vdots&\vdots&\ddots&\\
0&0&0&\cdots&1
\end{vmatrix}\qquad \big(0\leq j\leq[\frac{i}{2}]\big)\\
=&c_{i-j}c_j-c_{i-j+1}c_{j-1}.\end{split}\ee

The expression (\ref{schurspecialcase2}) shall play a decisive role in establishing the lower and upper bounds of the Chern classes/numbers in our Theorem \ref{theoremchernnumber}, which was partly inspired by \cite{DPS} as the special case $$S_{(i-1,1,0,\ldots,0)}(c_1,\ldots,c_r)=c_{i-1}c_1-c_i$$ has been observed and applied in \cite[Coro. 2.6]{DPS}.

Now we recall the notions of nonnegativity and strong nonnegativity for real $(p,p)$-forms (\cite[Ch. 3, \S 1.A]{De1}). Recall that a real $(p,p)$-form $\varphi$ ($1\leq p\leq n$) on a compact complex manifold $M^n$ is called \emph{nonnegative}, denoted by $\varphi\geq 0$, if for any point $x\in M$ and any $(1,0)$-type tangent vectors $X_1,\ldots,X_p$ at $x$, we have
$$(-\sqrt{-1})^{p^2}\varphi(X_1,\ldots,X_p,\overline{X_1},\ldots,
\overline{X_p})\geq 0.$$

A real $(p,p)$-form $\varphi$ is called \emph{strongly nonnegaive}, denoted by $\varphi\geq_{\text{s}}0$, if it can be written as
$$\varphi=(\sqrt{-1})^{p^2}
\sum_i\psi_i\wedge\overline{\psi_i},$$
where these $\psi_i$ are $(p,0)$-forms. A strongly nonnegative form is clearly nonnegative.

Two $(p,p)$-forms $\varphi_1$ and $\varphi_2$ are said to be $\varphi_1\leq\varphi_2$ (resp. $\varphi_1\leq_{\text{s}}\varphi_2$) if $\varphi_2-\varphi_1\geq0$ (resp. $\geq_{\text{s}}0$). It is clear that
the sum of two (strongly) nonnegative forms of the same degree is still (strongly) nonnegative. In general the product of two nonnegative forms may not be nonnegative (cf. \cite{BP}). Nevertheless, it is obvious that the product of two strongly nonnegative forms is still strongly nonnegative.

\section{Main technical results}\label{sectionmainresults}
Bott and Chern noticed that all the Chern forms $c_i(E,h)$ $(1\leq i\leq n)$ of a Bott-Chern nonnegative Hermitian vector bundle $(E^r,h)\longrightarrow M^n$ are strongly nonnegative (\cite[p. 91]{BC}). %This fact was obtained again by Matsushima-Stoll in \cite{MS} in the case of globally generated holomorphic vector bundles by resorting to the universal holomorphic vector bundles of complex Grassmannians.
The following proposition, which is our first main technical observation in this article, states that indeed many more real forms involving in the Chern forms, including $c_i(E,h)$ themselves, are strongly nonnegative.

\begin{proposition}\label{theoremschur}
Suppose that $(E^r,h)\longrightarrow M^n$ is a Hermitian vector bundles with $h\geq_{\text{BC}}0$. Then the following closed real $(i,i)$-forms
\be S_{\lambda}\big(c_1(E,h),\ldots,c_r(E,h)\big),\qquad
(\forall~\lambda\in\Gamma(i,r),~\forall~1\leq i\leq n)\nonumber\ee
are strongly nonnegative. In particular, by {\rm(}\ref{schurspecialcase1}{\rm)} and {\rm(}\ref{schurspecialcase2}{\rm)} the closed real $(i,i)$-forms
\begin{eqnarray}\label{theoremschurspecialcase}
\left\{ \begin{array}{ll}
c_i(E,h)&(1\leq i\leq n)\\
~\\
c_{i-j}(E,h)c_j(E,h)-c_{i-j+1}(E,h)c_{j-1}(E,h)&(1\leq i\leq n, 1\leq j\leq[\frac{i}{2}])
\end{array} \right.
\end{eqnarray}
are strongly nonnegative.
\end{proposition}

For later simplicity we denote by
$$c_{\lambda}(E,h):=\prod_{j=1}^i
c_{\lambda_j}(E,h),\qquad\forall~\lambda=(\lambda_1,\ldots,\lambda_i)\in\Gamma(i,r),$$
$$c_{\lambda}[E]:=\int_M\prod_{j=1}^n
c_{\lambda_j}(E,h)\in\mathbb{Z},\qquad\forall~\lambda=(\lambda_1,\ldots,\lambda_n)\in\Gamma(n,r),$$
and
$$c_{\lambda}(E):=\prod_{j=1}^i
c_{\lambda_j}(E)\in H^{2i}(M;\mathbb{Z}),\qquad\forall~\lambda=(\lambda_1,\ldots,\lambda_i)\in\Gamma(i,r).$$

We are now able to deduce our second technical result, which says that the strongly nonnegative forms in (\ref{theoremschurspecialcase}) in fact implies that the self-products of Chern forms and thus the Chern numbers can be bounded below and above as follows.
\begin{theorem}\label{theoremchernnumber}
Suppose that $(E^r,h)\longrightarrow M^n$ is a Hermitian vector bundle with $h\geq_{\text{BC}}0$. Then as real $(i,i)$-forms, $c_{\lambda}(E,h)\geq_{\text{s}}0$, and bounded below and above respectively by $c_i(E,h)$ and $c_1^i(E,h)$:
\be0\leq_{\text{s}} c_i(E,h)\leq_{\text{s}} c_{\lambda}(E,h)\leq_{\text{s}} c_1(E,h)^i,\qquad\forall~\lambda\in\Gamma(i,r),~\forall~1\leq i\leq n.\nonumber\ee
Consequently the Chern numbers satisfy
\be0\leq c_n[E]\leq c_{\lambda}[E]\leq c_1^n[E],\qquad\forall~\lambda\in\Gamma(n,r).\nonumber\ee
\end{theorem}
\begin{proof}
For any $\lambda=(\lambda_1,\ldots,\lambda_i)\in\Gamma(i,r)$, repeated use of (\ref{theoremschurspecialcase}) leads to
\be
\begin{split}
0\leq_{\text{s}} c_i(E,h)&\leq_{\text{s}} c_{i-1}(E,h)c_{1}(E,h)\\
&\leq_{\text{s}} c_{i-2}(E,h)c_{2}(E,h)\\
&\leq_{\text{s}}\cdots\\
&\leq_{\text{s}}\text{$c_{i-\lambda_1}(E,h)c_{\lambda_1}(E,h)$,\qquad if $\lambda_1\leq[\frac{i}{2}]$,}
\end{split}\nonumber\ee
or
$$0\leq_{\text{s}} c_i(E,h)\leq_{\text{s}} c_{i-1}(E,h)c_{1}(E,h)\leq_{\text{s}}\cdots\leq_{\text{s}} c_{\lambda_1}(E,h)c_{i-\lambda_1}(E,h),\qquad \text{if $\lambda_1>[\frac{i}{2}]$.}$$

This implies that in any case we have
\be\label{lowerbound1}0\leq_{\text{s}} c_i(E,h)\leq_{\text{s}} c_{\lambda_1}(E,h)c_{i-\lambda_1}(E,h)\ee

Similarly,
\be\label{lowerbound2}\begin{split}
&c_{i-\lambda_1}(E,h)\leq_{\text{s}} c_{\lambda_2}(E,h)c_{i-\lambda_1-\lambda_2}(E,h),\\
&c_{i-\lambda_1-\lambda_2}(E,h)\leq_{\text{s}} c_{\lambda_3}(E,h)c_{i-\lambda_1-\lambda_2-\lambda_3}(E,h),\\
&\cdots.
\end{split}\ee

Note that all these forms discussed here are strongly nonnegative by Prop. \ref{theoremschur} and thus their products are still strongly nonnegative:
\be\label{lowerbound3}0\leq_{\text{s}} c_i(E,h)\leq_{\text{s}} \prod_{j=1}^i
c_{\lambda_j}(E,h)=c_{\lambda}(E,h).\ee

On the other hand, similar arguments yield
\be\label{upperbound1}
\begin{split}
c_{\lambda_j}(E,h)&\leq_{\text{s}} c_{\lambda_j-1}(E,h)c_{1}(E,h)\\
&\leq_{\text{s}} \big[c_{\lambda_j-2}(E,h)c_{1}(E,h)\big]c_{1}(E,h)\\
&\leq_{\text{s}}\cdots\\
&\leq_{\text{s}} c_{1}(E,h)^{\lambda_j}.
\end{split}\ee

Therefore,
\be\label{upperbound2}c_{\lambda}(E,h)=\prod_{j=1}^i
c_{\lambda_j}(E,h)\leq_{\text{s}}\prod_{j=1}^i
c_{1}(E,h)^{\lambda_j}=c_{1}(E,h)^i,\qquad\forall~\lambda\in\Gamma(i,r).\ee

\end{proof}
\begin{remark}
This theorem is indeed partly inspired by an observation in \cite[Coro. 2.6]{DPS}, where they noticed that $S_{(i-1,1,0,\ldots,0)}=c_1c_{i-1}-c_i$ and thus gave the upper bound $c_1^i$ for $c_{\lambda}$ as
we have done in (\ref{upperbound1}) and (\ref{upperbound2}) in the context of nef vector bundles on compact K\"{a}hler manifolds. The reason why the deductions in (\ref{upperbound1}) and (\ref{upperbound2}) are true in their situation is due to the facts that nonnegative real $(1,1)$-forms (in particular for $c_1(E,h)$) are strongly nonnegative and the product of a strongly nonnegative form with a nonnegative form is still nonnegative (\cite[p. 132]{De1}). Nevertheless, in their context the inequality (\ref{lowerbound3}) is no longer true and so their is \emph{no} respectively lower bound in their situation. Indeed we shall see in Section \ref{sectionconsequence3} that it is this lower bound that plays key roles in many related applications.
\end{remark}

\section{Examples}\label{sectionexample}
In this section we shall illustrate by some examples that in many important situations the Bott-Chern nonnegativity can be satisfied. Although these examples and facts are well-known to experts, for the reader's convenience, we still outline some necessary details and/or point out some references.

Recall that a holomorphic vector bundle $E\longrightarrow M$ is called \emph{globally generated} if the global holomorphic sections of $E$ span the fiber over each point of $M$. If $M$ is compact then $H^0(M,E)$, the complex vector space consisting of holomorphic sections of $M$, is finite-dimensional. Then the property of being globally generated implies that the following bundle sequence
\be\label{globallygenratedsequence}
\begin{split}
0\longrightarrow\text{ker}(\varphi)\longrightarrow M\times
H^0(M,E)&\stackrel{\varphi}{\longrightarrow}E\longrightarrow 0\\ (x,s)&\longmapsto\big(x,s(x)\big).\end{split}
\ee
is exact. This means that a globally generated vector bundle over a compact complex manifold can be realized as a quotient bundle of a trivial vector bundle. Then the fact that
 \be\label{globallyimplyBC}
\begin{split}
&\text{\emph{any globally generated vector bundle over a compact complex manifold}}\\
&\text{\emph{admits a Bott-Chern nonnegative Hermitian metric}}
\end{split}\ee
follows from the following general result.
\begin{example}\label{examplegloballygenerated}
If $(E^r,h)\longrightarrow M$ is a Hermitian vector bundle over a compact complex manifold $M$ and $S$ is an $l$-dimensional holomorphic subbundle of $E$ ($l<r$) and thus $Q:=E/S$ is a $(r-l)$-dimensional holomorphic quotient bundle of $E$, i.e., we have the following short bundle exact sequence:
\be0\longrightarrow S\longrightarrow E\longrightarrow Q\longrightarrow 0.\nonumber\ee

The Hermitian metric $h$ on $E$ naturally induces a metric on $S$ and on $Q$. Bott and Chern noticed that (\cite[p. 91]{BC}), under any unitary frame $\{e_1,\ldots,e_r\}$ of $E$ such that $\{e_1,\ldots,e_l\}$ and $\{e_{l+1},\ldots,e_r\}$ are unitary frames of $S$ and $Q$ respectively, the curvature matrices of $E$ and $Q$, denoted by $\Omega_E$ and $\Omega_Q$, are related by (cf. \cite[p. 79]{GH}, \cite[p. 275]{De1})
$$\Omega_E\big|_Q=\Omega_Q-A\wedge\overline{A^t},$$
where $(\cdot)\big|_Q$ denotes the restriction to $Q$ and $A$ is a matrix whose entries are $(1,0)$-forms representing the second fundamental form of the subbundle $S$ in $E$.

This property is usually called ``\emph{curvature increases in holomorphic quotient bundles}''.
This means that if $E$ can be equipped with a \emph{flat} Hermitian metric as in (\ref{globallygenratedsequence}), which implies that $\Omega_E\big|_Q=0$, then the quotient bundle $Q$ admits a Bott-Chern nonnegative Hermitian metric. Thus the fact (\ref{globallyimplyBC}) follows.
\end{example}

Our principal concerns in this article are compact complex manifolds whose holomorphic tangent or cotangent bundles, denoted for simplicity in the sequel by $T$ and $T^{\ast}$ respectively, are globally generated. To each kind of them there is a well-known result, which we record in the following two examples for our later purpose.

\begin{example}\label{examplehomogeneous}
Recall that a compact connected complex manifold $M$ is called \emph{homogeneous} if the holomorphic automorphism group of $M$, denoted by $\text{Aut}(M)$, acts transitively on it. It is well-known that \emph{a compact connected complex manifold $M$ has globally generated $T_M$ if and only if it is homogeneous}.
%The reason is quite simple: $T_M$ being globally generated is equivalent to the fact that each orbit of the action of $\text{Aut}(M)$ is open in $M$. Note that $M$ is equal to a disjoint union of orbits. Then connectedness implies that there exists precisely one orbit, i.e., the action is homogenous.

This particularly implies that the holomorphic tangent bundles of Hermitian symmetric spaces, projective-rational manifolds and tori are globally generated.
\end{example}

\begin{example}\label{examplehomoimmersed}
\emph{A compact connected K\"{a}hler manifold has globally generated $T^{\ast}$ if and only if it can be holomorphically immersed into some complex torus}.

This result should be due to Matsushima and Stoll (\cite[p. 100]{MS} or \cite[p. 606]{Ma}), at least to the author's best knowledge. The basic idea is to consider the Albanese map related to this K\"{a}hler manifold. A more compact proof can be found in \cite[p. 271]{Sm}. Note that the term ``globally generated" was called ``ample" and ``weakly ample'' respectively in \cite{MS} and \cite{Sm}. Also note that a compact complex manifold holomorphically immersed into some torus is automatically K\"{a}hler as the restriction of the obvious K\"{a}hler metric on the complex torus is still K\"{a}hler. This implies that \emph{any compact connected complex manifold holomorphically immersed into some complex torus has globally generated $T^{\ast}$}.
\end{example}

%For compact complex manifolds whose cotangent bundles are globally generated we have the following well-known fact:
The notion of \emph{nefness} of line/vector bundles over projective algebraic manifolds is well-known and has been extended to general compact complex manifolds by Demailly, which can be viewed as a limiting case of ampleness. For its definition and basic properties we refer to \cite[\S 1]{De} or \cite[\S 1]{DPS}. Before ending this section, we need to point out an implication, which will be used in the next section, that the Bott-Chern nonnegativity implies nefness.
Indeed it turns out that Griffiths positivity implies ampleness (\cite[Theorem B]{Gr}, \cite[p. 206]{Zh}). Note that ampleness was called cohomological positivity in \cite{Gr}. Taking the limiting case we know that Griffiths nonnegativity implies nefness. Then the above-mentioned implication follows from the fact that the Bott-Chern nonnegativity in fact implies the Griffiths nonnegativity.
The latter fact is well-known to experts. For the reader's convenience we reord it in the following example and include a proof for the reader's convenience as well as for completeness.
\begin{example}\label{examplebottchernimpliesnefness}
We have the following implications:
\be\label{bottchernimpliesnefness}
\text{Bott-Chern nonnegativity}\Longrightarrow\text{Griffiths nonnegativity}\Longrightarrow\text{Nefness}.\ee
\end{example}
\begin{proof}
We show the first implication in (\ref{bottchernimpliesnefness}).
Under a local local coordinate $(z^1,\ldots,z^n)$ , the entries $\Omega^i_j$ in the curvature matrix $\Omega$ with respect to some unitary frame field can be written as $$\Omega^i_j=R^i_{jp\bar{q}}\text{d}z^p\wedge\text{d}\overline{z^q}.$$ The Hermitian metric $h$ on $E$ is called \emph{Griffiths nonnegative} (\cite[p. 181]{Gr}) if, at any point of $M$, we have
$$\sum_{i,j,p,q}R^i_{jp\bar{q}}\xi^j\eta^p\overline{\xi^i}\overline{\eta^q}\geq 0$$
for any $\xi=(\xi^1,\ldots,\xi^r)\in\mathbb{C}^r,$ and any
$\eta=(\eta^1,\ldots,\eta^n)\in\mathbb{C}^n.$

Now we assume that $h\geq_{\text{BC}}0$, i.e.,
\be(\Omega^i_j)=A\wedge\overline{A^t},\qquad A=(\sum_pT^{(p)}_{ij}\text{d}z^p).\nonumber\ee
This implies that
$$R^i_{jp\bar{q}}=\sum_kT^{(p)}_{ik}\overline{T^{(q)}_{jk}}$$
and thus
\be\begin{split}\sum_{i,j,p,q}R^i_{jpq}\xi^j\eta^p\overline{\xi^i}\overline{\eta^q}
=&\sum_{k,i,j,p,q}T^{(p)}_{ik}\overline{T^{(q)}_{jk}}\xi^j
\eta^p\overline{\xi^i}\overline{\eta^q}\\
=&\sum_{k,i,j,p,q}(T^{(p)}_{ik}\overline{\xi^i}\eta^p)
(\overline{T^{(q)}_{jk}\overline{\xi^j}\eta^q})\\
=&\sum_k|\sum_{i,p}T^{(p)}_{ik}\overline{\xi^i}\eta^p|^2\geq0.\end{split}\nonumber\ee
\end{proof}

\section{The first application}\label{sectionconsequence1}
In this and the next two sections we shall give three kinds of related applications, whose main contents have been briefly described in the Introduction.

Bott and Chern observed that (\cite[p. 91]{BC}) the Chern forms of Bott-Chern nonnegative Hermitian vector bundles are all strongly nonnegative and thus it is immediate that their Chern numbers are all nonnegative.

Our \emph{first direct application} of our Theorem \ref{theoremchernnumber} yields the following improvements on the aforementioned nonnegativity by giving lower and upper bounds.
\begin{theorem}\label{theoremapp1}
Suppose $(E^r,h)\longrightarrow M^n$ is a Bott-Chern nonnegative Hermitian vector bundle. Then all their Chern numbers are nonnegative and bounded below and above by $c_n[E]$ and $c_1^n[E]$:
\be\label{nonnegative1}0\leq c_n[E]\leq c_{\lambda}[E]\leq c_1^n[E],\qquad\forall~\lambda\in\Gamma(n,r).\ee
This particularly holds for globally generated vector bundles. In particular, if the Chern number $c_1^n[E]=0$, then all the Chern numbers vanish.
\end{theorem}

Compact connected homogeneous complex manifolds have been well-studied and it is also a classical result that their Chern numbers are nonnegative (\cite{Go},\cite{BR},\cite{GR}, \cite{MS}). The nonnegativity of the signed Chern numbers of compact connected complex manifolds holomorphically immersed into complex tori has also been obtained by Matsushima and Stoll (\cite[Thm 6.5]{MS}, \cite[p. 607]{Ma}). %These examples, as we have seen in Examples \ref{examplehomogeneous} and \ref{examplecotangentgloballygenerated}, satisfy the assumptions in Theorem \ref{theoremapp1} and thus we have
The following corollary improves the nonnegativity results for Chern numbers/signed Chern numbers of these manifolds by giving lower and upper bounds respectively.
\begin{corollary}\label{coromain}
\begin{enumerate}
\item
The Chern numbers of any compact connected homogeneous complex manifold $M^n$ are nonnegative and bounded below and above by $c_n[M]$ and $c_1^n[M]$:
\be\label{nonnegative2}0\leq c_n[M]\leq c_{\lambda}[M]\leq c_1^n[M],\qquad\forall~\lambda\in\Gamma(n,n).\ee

\item
The signed Chern numbers of any compact connected complex manifold $M^n$ whose holomorphic cotangent bundle is globally generated are nonnegative and bounded below and above by $(-1)^nc_n[M]$ and $(-1)^nc_1^n[M]$:
\be\label{theoremapp13}
0\leq(-1)^nc_n[M]\leq(-1)^nc_{\lambda}[M]\leq (-1)^nc_1^n[M],\qquad\forall~\lambda\in\Gamma(n,n).\ee
This particularly holds for compact connected complex manifolds holomorphically immersed into complex tori.
In particular, in these two cases, if the Chern number $c_1^n[M]=0$, then all the Chern numbers of $M$ vanish.
\end{enumerate}
\end{corollary}
%\begin{remark}
%The consideration of compact connected K\"{a}hler manifolds with globally generated holomorphic cotangent bundles can be traced back to Bochner (\cite{Bo}), where he showed the nonnegativity of the signed Euler characteristic.
%\end{remark}

%As mentioned in the Introduction, Demailly, Peternell and Schneider has showed the nonnegativity of the Schur polynomials in the Chern classes for nef holomorphic vector bundles over compact K\"{a}hler manifolds (\cite[Theorem 2.5]{DPS}) and obtained the upper bound $c_1^i$ for any $c_{\lambda}$, $\lambda\in\Gamma(i,r)$. Now by our arguments and remarks in Theorem \ref{theoremchernnumber} their result can be complemented by adding a lower bound as follows.
%\begin{corollary}\label{coroappl1}
%Suppose $E^r\longrightarrow(M,\omega)$ is a nef holomorphic vector bunde over a compact K\"{a}hler manifold with K\"{a}hler form $\omega$. Then we have
%$$0\leq\int_Mc_{i}(E)\wedge[\omega]^{n-i}\leq
%\int_Mc_{\lambda}(E)
%\wedge[\omega]^{n-i}\leq\int_Mc_1^i(E)
%\wedge[\omega]^{n-i},\qquad\forall~\lambda\in\Gamma(i,r).$$In particular, if the Chern number $c_1^n[E]=0$ (resp. $c_n[E]>0$), then all the Chern numbers of $E$ vanish (resp. are positive).\end{corollary}

\section{The second application}\label{sectionconsequence2}
Now we come to our second application. First let us recall some more notation. For a compact connected complex manifold $M^n$, the \emph{algebraic dimension} of $M$, denoted by $a(M)$, is defined to be the maximal number of meromorphic functions on $M$ that can be algebraically independent. In other words, $a(M)$ is the transcendence degree of function field of $M$ over $\mathbb{C}$, which is no more than $\text{dim}_{\mathbb{C}}(M)$. A compact complex manifold $M^n$ is called \emph{Moishezon} if its algebraic dimension attains the maximum: $a(M^n)=n$, in which case it was investigated in detail by Moishezon (\cite{Mo}). We refer the reader to \cite[Ch. 2]{MM} for a detailed presentation of the materials related to Moishezon manifolds.
%A deep theorem of Moishezon states that (cf. \cite[p. 95]{MM})
%\be\label{moishezonresult}\text{\emph{a Moishezon manifold is K\"{a}hler if and only if it is projective algebraic.}}\ee

The starting point of our second application is to recall two fundamental results due to Siu-Demailly and Demailly respectively.
%A major result of Matsushima-Stoll in \cite{MS} says that (\cite[Thms 5,5, 5.6]{MS} or \cite[p. 600]{Ma}) if a globally generated vector bundle $E^r\longrightarrow M^n$ satisfies $0\neq c_{\lambda}(E)\in H^{2i}(M,\mathbb{Z})$ for some $\lambda\in\Gamma(i,r)$, then the algebraic dimension $a(M)\geq i.$ In particular, \emph{if some Chern number of a globally generated vector bundle over $M$ is nonzero, then $M$ is Moishezon.}
%\be\label{matsushimastollresult}
%\begin{split}
%&\text{\emph{if some Chern number of a globally generated holomorphic vector bundle over $M$ }}\\
%&\text{\emph{is nonzero, then $M$ is Moishezon.}}
%\end{split}
%\ee
%This result, together with Moishezon's result that (cf. \cite[p. 95]{MM})
 %\be\label{Moishezontheorem}
%\text{\emph{a Moishezon manifold is K\"{a}hler if and only if it is projective algebraic,}}\ee leads to another main theorem of \cite{MS} (\cite[p. 94, Main Thm]{MS} or \cite[p. 600]{Ma}) saying that \emph{if some Chern number of a globally generated vector bundle over a compact K\"{a}hler manifold is nonzero, then this manifold is projective algebraic}.

The following result is due to Siu and Demailly (cf. \cite[p. 96]{MM}), which solved the Grauert-Riemenschneider conjecture and gives an effective criterion for a compact connected complex manifold to be Moishezon.
\begin{lemma}[Siu-Demailly]\label{siudemaillyresult}
A compact connected complex manifold $M$ is Moishezon provided that there exists a holomorphic line bundle $L$ on $M$ such that $c_1(L)$ is quasi-positive.
\end{lemma}
Here $c_1(L)$ is called \emph{quasi-positive} if there exists a real $(1,1)$-form representing $c_1(L)$ which is nonnegative everywhere and positive somewhere.

Another basic result due to Demailly (\cite[Coro. 1.6]{De} or \cite[Thm 4.1]{DPS}) tells us that
\begin{lemma}[Demailly]\label{demaillyresult}
A Moishezon manifold with nef tangent bundle is projective.
\end{lemma}

%Our strating observation in this second application is to give a new proof of the above-mentioned result due to Matsushima-Stoll by combining our Prop. \ref{theoremschur} and Thm. \ref{theoremapp1} with Siu-Demailly's result (\ref{siudemaillyresult}).\begin{theorem}[Matsushima-Stoll]\label{theoremproj1}Suppose $E^r\longrightarrow M^n$ is a globally generated vector bundle over a compact complex (resp. K\"{a}hler) manifold . If some Chern number of $E\longrightarrow M$ is nonzero, then $M$ is a Moishezon (resp. projective algebraic) manifold.\end{theorem}\begin{proof}By (\ref{globallyimplyBC}) there exists a Hermitian metric $h$ on $E$ such that $h\geq_{BC}0$. Therefore Prop. \ref{theoremschur} says that the first Chern form $c_1(E,h)$ is nonnegative as a real $(1,1)$-form. If by the assumption some Chern number of $M$ is nonzero, then the Chern number$$c_1^n[E]=\int_Mc_1(E,h)^n>0$$ by \ref{nonnegative1}. This means that $c_1(E,h)$ must be positive somewhere and hence the first Chern class $[c_1(E,h)]\in H^{1,1}(M;\mathbb{Z})$ is quasi-positive. So the desired conclusion follows from the Siu-Demailly criterion (\ref{siudemaillyresult}).\end{proof}
Compact connected homogeneous complex manifolds have been well-studied and a well-known fact, which should be due to H.-C. Wang, states that (cf. \cite[Thm 6.2]{MS} or \cite[p. 601]{Ma}) they are projective provided that their \emph{Euler characteristic}, which equal to the top Chern number $c_n$, are nonzero. Our following result is a refinement of this classical fact.
\begin{theorem}\label{theoremproj2}
Suppose $M^n$ is a compact connected homogeneous complex manifold. Then $M$ is a projective manifold if some Chern number of $M$ is nonzero.
\end{theorem}
\begin{remark}
Note that our inequality (\ref{nonnegative2}) indeed implies that the non-vanishing of the Euler characteristic $c_n[M]$ is equivalent to the non-vanishing of \emph{all} the Chern numbers.
\end{remark}
\begin{proof}
If some Chern number $c_{\lambda}[M]\neq0$, then (\ref{nonnegative2}) implies that the Chern number \be\label{nonzero}c_1^n[M]>0.\ee
By our discussions in Section \ref{sectionexample} the holomorphic tangent bundle $T_M$ admits a Hermitian metric $h$ such that $(T_M,h)\geq_{BC}0$ and so Proposition \ref{theoremschur} tells us that the first Chern form $c_1(T_M,h)$ is nonnegative. This, together with (\ref{nonzero}), implies that the first Chern class of $M$, which is also the first Chern class of the anti-canonical line bundle $K_M^{-1}$, is quasi-positive. Thus $M$ is Moishezon via Lemma \ref{siudemaillyresult}. The implication (\ref{bottchernimpliesnefness}) tells us that $T_M$ is nef and so $M$ is projective by Lemma \ref{demaillyresult}.
\end{proof}

%\begin{remark}
%In the statements of Theorems \ref{theoremproj1} and \ref{theoremproj2}, the condition that all the Chern numbers vanish can be restated in the language of cobordism theory that this manifold represents the zero element in cobordism ring (\cite{St}).
%\end{remark}
Note that in Theorem \ref{theoremproj2} the condition of some Chern number be nonzero is \emph{only} sufficient and its converse part is \emph{not} true in general, i.e., even if a compact connected homogeneous complex manifold is projective, it may happen that its all Chern numbers vanish. For instance, the product of an abelian variety over $\mathbb{C}$ and a projective-rational manifold is projective but its Chern numbers are all zero. Indeed, a classical result of Borel and Remmert tells us that (\cite{BR})
%\begin{lemma}[Borel-Remmert]\label{BR}any compact connected homogeneous K\"{a}hler manifold is the product of a complex torus and a projective-rational manifold.\end{lemma}
\be\label{BorelRemmertresult}\begin{split}&\text{\emph{any compact connected homogeneous K\"{a}hler manifold is the product of a complex}}\\&\text{\emph{torus and a projective-rational manifold.}}\end{split}\ee

This result (\ref{BorelRemmertresult}) implies that the condition of some Chern number being nonzero is also necessary in Theorem \ref{theoremproj2} if we further impose \emph{simple-connectedness} on the manifolds in question and thus we have the following result.
\begin{theorem}\label{coroproj2}
Suppose $M^n$ is a simply-connected compact connected homogeneous complex manifold. Then the following four conditions are equivalent:
\begin{enumerate}
\item
The Chern number $c_1^n[M]\neq0$.

\item
Some Chern number of $M$ is nonzero.

\item
$M$ is a projective manifold.

\item
$M$ is a K\"{a}hler manifold.
\end{enumerate}
\end{theorem}
\begin{remark}
\begin{enumerate}
\item
The manifolds considered in Theorem \ref{coroproj2} were called \emph{$C$-spaces} by Wang (\cite{Wa}) and investigated in detail by many authors (\cite{BR}, \cite{Go}, \cite{GR}, \cite{Wa} etc). However, the contents in our Theorem \ref{coroproj2} should be completely new, at least to the author's best knowledge.
%\item
%In the conclusions in Theorems \ref{theoremproj1}, \ref{theoremproj2} and \ref{coroproj2}, when the Chern numbers in consideration are nonzero, they must be positive by (\ref{nonnegative2}) in Corollary \ref{coromain}.
\item
There are homogeneous complex structures on the product of odd-dimensional spheres: $M_{p,q}:=S^{2p+1}\times S^{2q+1}$ \big($(p,q)\neq(0,0)$\big), which are called Calabi-Eckmann manifolds (\cite{CE}) and whose Chern numbers are all zero. $M_{p,q}$ are indeed non-K\"{a}hler as its second Betti number is zero and so these examples match Theorem \ref{coroproj2} very well.
\end{enumerate}
\end{remark}

\section{The third application}\label{sectionconsequence3}
Our third application is to give an in-depth investigation on the structure of compact connected complex manifolds holomorphically immersed into complex tori in detail. As we have seen in Example \ref{examplehomoimmersed}, these manifolds are precisely those compact connected K\"{a}hler manifolds whose holomorphic cotangent bundles are globally generated. The structure of these manifolds was first investigated by Matsushima and his coauthors (\cite{HM}, \cite{Ma1}, \cite{MS}), Yau (\cite[Chapter 3]{Ya}) and Smyth (\cite{Sm}, \cite{NS}) etc around the same time. Later in \cite{Zhang} Zhang related the ampleness of the canonical line bundles of these manifolds to their Todd genera.

Before continuing, let us digress to recall the notion of \emph{Kodaira dimension} $\kappa(M)$  for compact complex manifolds $M$ (\cite[p. 132]{Zh}), which has several equivalent definitions. Here for our later purpose we adopt the following one. Denote by $K_M$ the canonical line bundle of $M$. It turns out that $\text{dim}_{\mathbb{C}}H^0(M, K_M^{\otimes m})$, the complex dimension of the holomorphic sections of $K_M^{\otimes m}$, has the following property: either $H^0(M, K_M^{\otimes m})=0$ for all $m\geq1$ or there exists an integer $0\leq \kappa(M)\leq\text{dim}_{\mathbb{C}}(M)$ and constants $0\leq C_1<C_2$ such that
\be\label{kodairadimension}C_1m^{\kappa(M)}\leq\text{dim}_{\mathbb{C}}H^0(M, K_M^{\otimes m})\leq C_2m^{\kappa(M)},\qquad m>>0,\ee
which means $\text{dim}_{\mathbb{C}}H^0(M, K_M^{\otimes m})$ grows at a rate of $m^{\kappa(M)}$.
In the former case $\kappa(M):=-\infty$. $M$ is called \emph{of general type} if $\kappa(\cdot)$ attains its maximum: $\kappa(M)=\text{dim}_{\mathbb{C}}(M)$.
%It is an elementary fact between the Kodaira dimension $\kappa(M)$ and the algebraic dimension $a(M)$ that (\cite[p. 135]{Zh}):
%\be\label{kodairaalgebraic}\kappa(M)\leq a(M)\leq\text{dim}_{\mathbb{C}}(M).\ee

Denote by $\text{Aut}_0(M)$ the identity component  of the holomorphic automorphism group of a compact complex manifold $M$, which is a connected complex Lie group. The most fundamental structure of compact connected complex manifolds holomorphically immersed into complex tori is the following result.
\begin{theorem}\label{structureresult}{\rm(}\cite[Prop. 1]{Ma1}, \cite[Th. 5]{Ya}, \cite[Th. 1]{Sm}{\rm)}
Suppose $M^n$ is a compact connected complex manifold holomorphically immersed into some complex torus with $\text{Aut}_0(M)\neq\{0\}$. Then $\text{Aut}_0(M)$ is a complex torus and acts freely on $M$. Moreover, the quotient manifold $N:=M/\text{Aut}_0(M)$ is also a compact connected complex manifold and can be holomorphically immersed into some complex torus and $\text{Aut}_0(N)=\{0\}$. Consequently, $M$ is a holomorphic principal complex torus bundle over some compact connected complex manifold $N$ also holomorphically immersed into some complex torus with $\text{Aut}_0(N)=\{0\}.$
\end{theorem}

This structure theorem has several applications in \cite{Ma1}, \cite{Ya} and \cite{Sm}. Let us record several of them and some results in \cite{Zhang} related to our next result in the following theorem.
\begin{theorem}\label{structureresult2}
Suppose $M^n$ is a compact connected complex manifold holomorphically immersed into some complex torus. Then
\begin{enumerate}
\item
\rm(\cite[p. 238]{Ya}\rm) $M$ is a (possibly trivial) torus bundle over a compact K\"{a}hler manifold of general type.
\item
(\cite[p. 278]{Sm}) The following three conditions are equivalent:
\be\label{Sm}\text{Aut}_0(M)=\{0\}
\Longleftrightarrow c_1^n[M]\neq0\Longleftrightarrow c_n[M]\neq0.\ee

\item
(\cite[Thm 3, Coro. 3]{Zhang}) The following three conditions are equivalent:
\be\label{Zhang}\text{$K_M$ is ample}
\Longleftrightarrow (-1)^n\text{td}(M)>0\Longleftrightarrow (-1)^nc_n[M]>0,\ee
where $\text{td}(M)$ is the Todd genus of $M$, which by definition is the alternating sum of the Hodge numbers $$\text{td}(M)=\sum_{q}(-1)^qh^{0,q}(M)$$
and the Hirzebruch-Riemann-Roch theorem says that it is a ratinally linear combination of Chern numbers.
\end{enumerate}
\end{theorem}

We can now give our third application, which relates
$\text{Aut}_0(M)$, all the Chern numbers of $M$, the Kodaira dimension $\kappa(M)$, the canonical line bundle $K_M$, and its torus bundle structure and particularly improves various results summarized in Theorem \ref{structureresult2}.
\begin{theorem}\label{theoremappl3}
Suppose $M^n$ is a compact connected complex manifold holomorphically immersed into some complex torus, or equivalently, $M^n$ is a compact connected K\"{a}hler manifold with globally generated holomorphic cotangent bundle. Then the following six conditions are equivalent:
\begin{enumerate}
\item
$\text{Aut}_0(M)=\{0\}$,
\item
all the signed Chern numbers $(-1)^nc_{\lambda}[M]$ are strictly positive,
\item
$M$ is of general type,
\item
the canonical line bundle $K_M$ is ample,
\item
the signed Todd genus is positive: $(-1)^n\text{td}(M)>0$,
\item
$M$ cannot be realized as a total space of some nontrivial torus bundle;\\

and the following six conditions are equivalent:\\

\item
$\text{Aut}_0(M)\neq\{0\}$,
\item
all the Chern numbers of $M$ vanish,
\item
the Kodaira dimension $\kappa(M)<n$,
\item
the canonical line bundle $K_M$ is not ample,
\item
the Todd genus $\text{td}(M)$ vanishes,
\item
$M$ is a nontrivial holomorphic principal complex torus bundle over a compact connected K\"{a}hler manifold with ample canonical line bundle.
\end{enumerate}
Moreover, under {\rm(}any one of{\rm)} the first six equivalent conditions, $M$ is projective.
\end{theorem}

A direct corollary of Theorem \ref{theoremappl3} is the following result, which tells us that whether or not the Chern numbers and Todd genus vanish is simultaneous.
\begin{corollary}\label{coroapp3}
Suppose $M^n$ is a compact connected complex manifold holomorphically immersed into some complex torus. Then%, or equivalently, $M^n$ is a compact connected K\"{a}hler manifold with globally generated holomorphic cotangent bundle.
\begin{enumerate}
\item
if the Todd genus $\text{td}(M)=0$, then all the Chern numbers vanish;

\item
if the Todd genus $\text{td}(M)\neq0$, then it has sign $(-1)^n$ and all the Chern numbers are nonzero and have the same sign $(-1)^n$.
\end{enumerate}
In particular, all Chern numbers either vanish or are nonzero and have the sign $(-1)^n$.
\end{corollary}

\emph{Proof of Theorem \ref{theoremappl3}.}
\begin{proof}
First we notice that the condition $(4)$ implies that $M$ is projective by the Kodaira embedding theorem, from which the last conclusion in Theorem \ref{theoremappl3} follows.

``$(2)\Rightarrow(6)$": If $M$ can be realized as a total space of some nontrival torus bundle, then $c_n[M]=0$ as for fiber bundles the Euler characteristic is multiplicative.

``$(6)\Rightarrow(1)$": This follows from Theorem \ref{structureresult}.

``$(1)\Leftrightarrow(2)$": This follows from (\ref{theoremapp13}) and (\ref{Sm}).

``$(2)\Rightarrow(3)$": The proof is a combination of a Kodaira-type vanishing theorem and the Hirzebruch-Riemann-Roch theorem.

By Example \ref{examplehomoimmersed} we can choose a Hermitian metric $h$ on $T_M^{\ast}$ such that $(h,T_M^{\ast})\geq_{\text{BC}}0$. Then (\ref{theoremschurspecialcase}) tells us that $c_1(T_M^{\ast},h)\geq 0$. This means that the Chern class
\be\label{vanishing1}c_1(K_M^{\otimes m})=mc_1(K_M)=m[c_1(T_M^{\ast},h)]\geq0,\qquad m\geq 1.\ee
Condition $(2)$ implies that
\be\label{vanishing2}\int_M\big(c_1(K_M^{\otimes m})\big)^n=m^n(-1)^nc_1^n[M]>0,\qquad m\geq 1.\ee

We claim that under the condition $(2)$ $M$ is projective. Indeed, as in the proof of Theorem \ref{theoremproj2}, $c_1(T_M^{\ast},h)\geq 0$ and $(-1)^nc_1^n[M]>0$ imply that $c_1(K_M)$ is quasi-positive and thus $M$ is Moishezon by Lemma \ref{siudemaillyresult}. A deep result of Moishezon himself tells us that a Moishezon K\"{a}hler manifold is projective (cf. \cite[p. 95]{MM}).

Then a combination of (\ref{vanishing1}), (\ref{vanishing2}), $M$ being projective and the Kodaira-Kawamata-Viehweg vanishing theorem for projective manifolds (\cite[p. 74, (3.10)]{Ko}) yields
\be\label{vanishing3}
H^q(M,K_M^{\otimes m})=0,\qquad q\geq1,~m>>0.\ee

Therefore the Hirzebruch-Riemann-Roch theorem tells us that (\cite{Hi})
\be\begin{split}
\text{dim}_{\mathbb{C}}H^0(M,K_M^{\otimes m})
=&\sum_{q=0}^n(-1)^q\text{dim}_{\mathbb{C}}H^q(M,K_M^{\otimes m})\qquad\big((\ref{vanishing3})\big)\\
=&\int_M\big[\text{Td}(M)\cdot\text{Ch}(K_M^{\otimes m})\big]\qquad(\text{Td: Todd class, Ch: Chern character})\\
=&\int_M\Big\{\big[1+\frac{1}{2}c_1(M)+\cdots\big]
\cdot\exp\big[-mc_1(M)\big]\Big\}\\
=&\frac{(-1)^nc_1^n[M]}{n!}m^n+O(m^{n-1}),\qquad(m>>0)
\end{split}\nonumber\ee
which, together with $(-1)^nc_1^n[M]>0$ under the condition $(2)$, implies (\ref{kodairadimension}) with the Kodaira dimension $\kappa(M)=n$, i.e., $M$ is of general type.

``$(3)\Rightarrow(4)$": %We first assert that $M$ is projective algebraic under the condition $(3)$. Indeed, $\kappa(M)=n$ implies that the algebraic dimension $a(M)=n$ by (\ref{kodairaalgebraic}) and so $M$ is Moishezon and thus projective algebraic.
This follows from a result of Ran (\cite[p. 176, Coro. 3]{Ra}), which says that a compact connected complex manifold of general type holomorphically immersed into some complex torus must have ample canonical line bundle.

``$(4)\Leftrightarrow(5)$": This follows from (\ref{Zhang}).

``$(4)\Rightarrow(2)$":
Note that the ampleness of $K_M$ implies that the signed Chern number $(-1)^nc_1^n[M]>0$, which, together with (\ref{Sm}), implies that $(-1)^nc_n[M]>0$. Then $(2)$ follows from our (\ref{theoremapp13}).
%is projective algebraic. If on the contrary the Euler characteristic $c_n[M]=0$, then a result of Howard and Matsushima (\cite[Th. 6]{HM} or \cite[p. 655]{Ma}) says that $M$ admits a holomorphic one-form whose zero-locus is empty. However, another major result of Zhang in \cite[Thm 1]{Zhang} tells us that the zero-locus of any holomorphic one-form on a projective algebraic manifold with ample canonical line bundle is non-empty, which leads to a contradiction. This means that $c_n[M]\neq0$ and thus proves condition $(2)$ by (\ref{theoremapp13}).

``$(8)\Rightarrow(11)$": Note that the Todd genus $\text{td}(M)$ is a rationally linear combination of Chern numbers.

 ``$(11)\Rightarrow(7)$": This follows from ``$(1)\Leftrightarrow(5)$".

 ``$(7)\Leftrightarrow(8)$": This follows from (\ref{theoremapp13}) and (\ref{Sm}).

``$(7)\Rightarrow(12)$": Theorem \ref{structureresult} tells us that, under the condition $\text{Aut}_0(M)\neq\{0\}$, $M$ is a nontrivial holomorphic principal complex torus bundle over a compact connected complex manifold $N$, where $N$ can also be holomrophically immersed some complex torus with $\text{Aut}_0(N)=\{0\}$. Then ``$(1)\Leftrightarrow(4)$" implies that $K_N$ is ample.

``$(12)\Rightarrow(7)$": Condition $(12)$ implies that the Euler characteristic $c_n[M]=0$ and thus $\text{Aut}_0(M)\neq\{0\}$ follows from ``$ (1)\Leftrightarrow(2)$".

Other equivalent relations among $(7)$ through $(12)$ are direct via those among $(1)$ through $(6)$.
\end{proof}

%\begin{remark}
%A recent result of Popa and Schnell (\cite{PS}) refined \cite[Thm 1]{Zhang} by showing that the zero-locus of any holomorphic one-form on a projective algebraic manifold of general type is non-empty. Thus we can also apply it to deduce ``$(3)\Rightarrow(2)$".
%We can easily see from Theorem \ref{structureresult} and the above proof that the condition $(8)$ in Theorem \ref{theoremappl3} can be strengthened to
%\be\label{strong}\begin{split}``&\emph{\text{$M$ is a nontrivial holomorphic principal complex torus bundle over some compact}} \\&\emph{\text{connected K\"{a}hler manifold $N$ with globally generated $T^{\ast}_N$ and trivial $\text{Aut}_0(N)$}}''\end{split}\eeor weakened to\be\label{weak}``\emph{$M$ can be realized as a total space of some nontrivial torus bundle",}\eeor can be stated as some condition, as we have done in Theorem \ref{theoremappl3}, provided that $$``(\ref{strong})\Longrightarrow\text{this condition}\Longrightarrow(\ref{weak})".$$The reason for our current choice of condition $(8)$ in Theorem \ref{theoremappl3} is to match Yau's result in Corollary \ref{structureresult2}. The condition $(4)$ in Theorem \ref{theoremappl3} can be similarly treated.
%\end{remark}

\section{Proof of Proposition \ref{theoremschur}}\label{sectionproofoftheorem1}
In this last section we shall prove Proposition \ref{theoremschur} and complete this article.

Let us begin by recalling some well-known facts about $GL_r(\mathbb{C})$-invariant polynomial functions. In what follows $GL_r(\mathbb{C})$ and $M_r(\mathbb{C})$ denote respectively the general linear group of order $r$ and the $r\times r$ matrix group, both over $\mathbb{C}$. A map $f: M_r(\mathbb{C})\rightarrow\mathbb{C}$ is called a \emph{$GL_r(\mathbb{C})$-invariant polynomial function of homogeneous degree $i$} if $f$ can be written as
\be\label{polynomialfunction1}
f\big(A=(T_{\alpha\beta})\big)
=\sum_{1\leq\alpha_j,\beta_j\leq r}p_{\alpha_1\cdots\alpha_i,
\beta_1\cdots\beta_i}
T_{\alpha_1\beta_1}\cdots T_{\alpha_i\beta_i}\ee
for $p_{\alpha_1\cdots\alpha_i,
\beta_1\cdots\beta_i}\in\mathbb{C}$ and satisfies
\be\label{polynomialfunction2}
f(BAB^{-1})=f(A),\qquad\forall~A\in M_r(\mathbb{C}),~\forall~B\in GL_r(\mathbb{C}).\ee

Denote by $I_i$ the complex linear space consisting of all $GL_r(\mathbb{C})$-invariant polynomial function of homogeneous degree $i$:
\be\label{Ii} I_i:=\{f:~M_r(\mathbb{C})\rightarrow\mathbb{C}~|~\text{$f$
satisfy (\ref{polynomialfunction1}) and (\ref{polynomialfunction2})}
\}\nonumber\ee
and the graded ring
$$I:=\bigoplus_{i\geq 0}I_i.$$

If we set
\be\det(tI_r+A)=:\sum_{i=0}^rc_i(A)\cdot t^{r-i},\nonumber\ee
then these $c_i(\cdot)$ are $GL_r(\mathbb{C})$-invariant polynomial functions of homogeneous degree $i$. In particular, $c_0(A)=1,$ $c_1(A)=\text{trace}(A)$ and $c_n(A)=\det(A)$. It is well-known that the graded ring $I$ is multiplicatively generated by $c_1,\ldots,c_r$, i.e., $$I=\mathbb{C}[c_1,\cdots,c_r].$$

Now if $(E^r,h)\longrightarrow M$ is a Hermitian holomorphic vector bundle and $(\Omega^i_j)$ is the curvature matrix of its Chern connection, then \big(recall (\ref{curvaturematrixtransformation}) and (\ref{det})\big)
$$c_i\big(\frac{\sqrt{-1}}{2\pi}(\Omega^i_j)\big)
=:c_i(E,h),\qquad(0\leq i\leq r)$$ are globally defined, closed, real-valued $(i,i)$-forms over $M$ and represent the $i$-th Chern classes of $E\longrightarrow M$.

Thus we have a natural graded ring homomorphism $\varphi$ sending $c_i$ to $c_i(E,h)$:
\be\label{naturalmap}
\begin{split}
 I=\bigoplus_{i\geq 0}I_i=\mathbb{C}[c_1,\cdots,c_r]&\xlongrightarrow{\varphi}\mathbb{C}[c_1(E,h),
\cdots,c_r(E,h)]\\
f&\longmapsto f\big(\frac{\sqrt{-1}}{2\pi}(\Omega^i_j)\big). \end{split}\ee

Griffiths first showed that (\cite[p. 242, (5.6)]{Gr}) any $f\in I_i$ can be written, which may not be unique, in the form
\be\label{griffithsexp}
\begin{split}
&f(A=\big(T_{ij})\big)\\
=&\sum_{\pi,\tau\in S_i,\rho=(\rho_1,\ldots,\rho_i)\in[1,r]^i}
p_{\rho,\pi,\tau}T_{\rho_{\pi(1)}\rho_{\tau(1)}}\cdots
T_{\rho_{\pi(i)}\rho_{\tau(i)}},\qquad p_{\rho,\pi,\tau}\in\mathbb{C},\end{split}\ee
where $S_i$ denotes the permutation group on $i$ objects.

The following definition was introduced by Griffiths (\cite[p. 242, (5.9)]{Gr}).
\begin{definition}[Griffiths]
A polynomial function $f\in I_i$ is called \emph{Griffiths positive} if it can be expressed in the form (\ref{griffithsexp}) with
\be\label{griffithsexp2}p_{\rho,\pi,\tau}=\sum_{j\in I}\lambda_{\rho,j}q_{\rho,j,\pi}\bar{q}_{\rho,j,\tau},\qquad\forall~\rho, \pi, \tau,\ee
for some real numbers $\lambda_{\rho,j}\geq 0$, some complex numbers $q_{\rho,j,\pi}$ and some finite set $J$.
\end{definition}

Denote by
$$\Pi_i:=\{f\in I_i~|~\text{$f$ are Griffiths positive}\}.$$
%It is clear from the definition that $\Pi_i$ is a convex cone and $\Pi_i\Pi_j\subset\Pi_{i+j}$.
%$$\Pi:=\oplus_{i\geq 0}\Pi_i$$
 %is called \emph{Griffiths positive cone} and graded by degree.

The following key fact relating Griffiths positive polynomials and Schur polynomials was observed in \cite[p. 54, Prop. A.3]{FL}, whose proof is built on the representation theory and a previously related result in \cite{UT}.
\begin{proposition}[Fulton-Lazarsfeld]
\be\label{FLprop}\Big\{\sum_{\lambda\in\Gamma(i,r)}
a_{\lambda}S_{\lambda}(c_1,\dots,c_r)~\big|~\text{all $a_{\lambda}\geq 0$}\Big\}=\Pi_i.\ee
\end{proposition}

Now we are in the position to prove Proposition \ref{theoremschur}.
\begin{proof}
Assume that the Hermitian holomorphic vector bundle $(E^r,h)\longrightarrow M^n$ satisfies $h\geq_{\text{BC}}0$, i.e., under the local coordinates $(z^1,\ldots,z^n)$, the curvature matrix $(\Omega^i_j)$ with respect to some unitary frame field can be written in the following form
\be(\Omega^i_j)=A\wedge\overline{A^t},\qquad A=(\sum_pT^{(p)}_{ij}\text{d}z^p).\nonumber\ee
This implies that
\be\label{curvatureexpr}\Omega^i_j=\sum_{k,p,q}T^{(p)}_{ik}
\overline{T^{(q)}_{jk}}\text{d}z^p\wedge\text{d}\overline{z^q}.\ee
Therefore, for each $\lambda\in\Gamma(i,r)$ and $1\leq i\leq n$, we have
\be\begin{split}
&S_{\lambda}\big(c_1(E,h),\ldots,c_r(E,h)\big)\\
=&\varphi\big(S_{\lambda}(c_1,\ldots,c_r)\big)\qquad\big((\ref{naturalmap})\big)\\
=&S_{\lambda}\Big(c_1\big(\frac{\sqrt{-1}}{2\pi}(\Omega^i_j)\big),\ldots,
c_r(\big(\frac{\sqrt{-1}}{2\pi}(\Omega^i_j)\big)\Big)\qquad\big((\ref{naturalmap})\big)\\
=&(\frac{\sqrt{-1}}{2\pi})^i\sum\lambda_{\rho,j}
q_{\rho,j,\pi}\bar{q}_{\rho,j,\tau}
\Omega^{\rho_{\pi(1)}}_{\rho_{\tau(1)}}\cdots
\Omega^{\rho_{\pi(i)}}_{\rho_{\tau(i)}}\qquad\big((\ref{griffithsexp}),~ (\ref{griffithsexp2}),~(\ref{FLprop})\big)\\
=&(\frac{\sqrt{-1}}{2\pi})^i\sum\lambda_{\rho,j}
q_{\rho,j,\pi}\bar{q}_{\rho,j,\tau}
T^{(r_1)}_{\rho_{\pi(1)}k_1}\overline{T^{(s_1)}_{\rho_{\tau(1)}k_1}}
\text{d}z^{r_1}\wedge\text{d}\overline{z^{s_1}}\cdots
T^{(r_i)}_{\rho_{\pi(i)}k_i}\overline{T^{(s_i)}_{\rho_{\tau(i)}k_i}}
\text{d}z^{r_i}\wedge\text{d}\overline{z^{s_i}}\qquad\big((\ref{curvatureexpr})\big)\\
=&(\frac{\sqrt{-1}}{2\pi})^i(-1)^{\frac{i(i-1)}{2}}
\sum_{\rho,j}\lambda_{\rho,j}\Big\{\big[\sum_{\pi,k,r}q_{\rho,j,\pi}T^{(r_1)}_{\rho_{\pi(1)}k_1}\cdots
T^{(r_i)}_{\rho_{\pi(i)}k_i}\text{d}z^{r_1}\wedge\cdots\wedge\text{d}z^{r_i}\big]\\
&\wedge\big[\sum_{\tau,k,s}\overline{q_{\rho,j,\tau}T^{(s_1)}_{\rho_{\tau(1)}k_1}
\cdots
T^{(s_i)}_{\rho_{\tau(i)}k_i}\text{d}z^{s_1}
\wedge\cdots\wedge\text{d}z^{s_i}}\big]\Big\}\\
=:&\frac{(\sqrt{-1})^{i^2}}{(2\pi)^i}\sum_{\rho,j}
\lambda_{\rho,j}\psi_{\rho,j}\wedge\overline{\psi_{\rho,j}}.
\end{split}\nonumber\ee
Now $\psi_{\rho,j}$ are $(i,0)$-forms and by the definition $\lambda_{\rho,j}\geq0$ \big(recall (\ref{griffithsexp2})\big), which means that the last expression is a strongly nonnegative $(i,i)$-form. This gives the desired proof and thus completes this article.
\end{proof}

\end{document}